\documentclass[draft]{amsart}

\usepackage{amssymb}
\usepackage{url}
\usepackage{amscd}

\theoremstyle{plain}
\newtheorem{theorem}{Theorem}

\newtheorem{corollary}[theorem]{Corollary}

\newtheorem{proposition}[theorem]{Proposition}
\newtheorem{definition}[theorem]{Definition}

\theoremstyle{remark}
\newtheorem{remark}{Remark}

\begin{document}

\renewcommand{\contentsname}{ }

\date{This version:  15.07.2010}

\title[Metric spaces with dilations]{Introduction to metric spaces with
dilations}

\author{Marius Buliga}

\address{Institute of Mathematics, Romanian Academy \\ 
P.O. BOX 1-764, RO 014700 \\ 
Bucure\c sti, Romania}
\email{Marius.Buliga@imar.ro}


\begin{abstract}
This paper gives a short introduction into the metric theory of 
spaces with dilations. 
\end{abstract}

\maketitle

\tableofcontents

\section{Introduction}

Metric spaces with dilations were introduced in \cite{buligadil1} under the 
name of "dilatation structures", then studied in a series of papers 
\cite{buligadil2} \cite{buligasr} \cite{buligaultra}. Very recently, in 
\cite{selivanova}, \cite{selvodopis}, the same object has been named
"(quasi)metric space with dilations". In the mentioned papers the authors 
extend the results from  \cite{buligadil1} to quasimetric spaces. 
We shall keep here this double denomination dilatation structure - metric space 
with dilations. 

Topological spaces with dilations were studied for the first time to my
knowledge in the paper \cite{buliga1}. In the paper \cite{buligaemergent} 
it is proved that the algebraic properties of  spaces with dilations are 
not based on metric notions, but in fact they hold for uniform spaces. 
Thus the generalization of Selivanova and Vodopyanov is not surprising at all, 
because quasimetric spaces are uniform topological spaces and this is all we 
need in order to deduce these mentioned algebraic properties. 
Another line of generalization was proposed in \cite{groupoids}, where normed 
groupoids and specific deformations of those were introduced. A particular
case is that of a trivial normed groupoid with a deformation induced by a dilatation 
structure. 

Finally, in the paper \cite{buligachar} we introduced length metric spaces with
dilations (length dilatation structures) and proved that regular sub-riemannian
spaces can be seen as such length dilatation structures. In the case of length
metric spaces with dilations we have to work with length functionals and 
study the gamma-convergence, or variational convergence, of length functionals, 
thus generalizing results obtained by Buttazzo,  De Pascale,  and Fragal\`a in 
\cite{buttazzo1}, or Venturini \cite{venturini}. 

In this paper I give a short introduction into these subjects, which could
serve as a basis for understanding more specialized results. 

In my opinion spaces with dilations could become a topic of intense studies.
Indeed, many examples studied in analysis in metric spaces are in fact spaces
with dilations and it seems that this supplementary algebraic-geometric 
structure which was recently identified could be a valuable tool for developing  
differential calculus or geometric measure theory in such spaces. For the moment
this subject has not been explored in combination with measure theory (for 
example on metric measured spaces, or in relation with optimal transportation). 
But it seems reasonable to expect that new results await just around the corner. 

\section{Metric spaces, distances, norms}

\begin{definition}
A metric space $(X,d)$ is a set $X$ endowed with a distance function  
$d: X \times X \rightarrow [0,+\infty)$.  
In the metric space $(X,d)$, the distance between two points $x, y \in X$ is $d(x,y) \geq 0$. 
The distance $d$ satisfies the following axioms: 
\begin{enumerate}
 \item[(i)] $d(x,y) = 0$ if and only if $x=y$, 
  \item[(ii)] (symmetry) for any $x, y \in X$ $d(x,y) = d(y,x)$, 
  \item[(iii)] (triangle inequality) for any $x, y, z \in X$ $d(x, z) \leq d(x,y) + d(y,z)$. 
  \end{enumerate}
The ball of radius $r>0$ and center $x \in X$ is the set 
$$ B(x,r) \, = \, \left\{ y \in X \mbox{  :  }  d(x,y) < r \right\}  \quad . $$
Sometimes we shall use the notation $\displaystyle B_{d}(x,r)$ for the ball of center $x$ and radius 
$r$ with respect to the distance $d$, in order to emphasize the dependence on the distance $d$. 
Any  metric space $(X,d)$ is endowed with  the topology generated by balls.  The notations 
 $\bar{B}(x,r)$ and  
$\displaystyle \bar{B}_{d}(x,r)$ are used for the closed ball centered at $x$, with radius $r$. 

A pointed metric space $(X, x, d)$ is a metric space $(X,d)$ with a chosen point $x \in X$. 
\label{dmetricspace}
\end{definition}

The notion of a metric space is not very old: it has been introduced by Fr\' echet in the paper 
[Sur quelques points du calcul fonctionnel, {\it Rendic. Circ. Mat. Palermo} {\em 22} (1906), 1Ð-74].

\subsection{Metric spaces, normed groups and normed groupoids} An obvious example of a metric space is $\displaystyle \mathbb{R}^{n}$ endowed with an euclidean distance, that is with a distance function induced by an euclidean norm: 
$$d(x, y) \, = \, \| x - y \| \quad . $$
In fact any normed vector space can be seen as a metric space. In order to define a distance 
from a norm, in a normed vector space, we only need the norm function and the abelian group 
structure of the vector space. (Later in this paper, he multiplication by scalars will provide us with the first example of a metric space with dilations).   This leads us to the introduction of  normed groups.  Let us give, in increasing generality, the definition of 
a normed group, then the definition of a normed groupoid. 

\begin{definition}
A normed group $(G, \rho)$ is a pair formed by: 
\begin{enumerate}
\item[-]  a group $G$, with the operation 
$(x, y) \in G \times G \mapsto xy$, inverse denoted by $\displaystyle x \in G \mapsto x^{-1}$  and neutral element denoted by $e$, 
\item[-]   a norm function $\rho : G \rightarrow [0, +\infty)$, which satisfies the following axioms: 
\begin{enumerate}
\item[(i)]  $\rho(x) = 0$ if and only if $x = e$, 
\item[(ii)] (symmetry) for any $x \in G$ $\displaystyle \rho(x^{-1} ) = \rho(x)$, 
\item[(iii)] (sub-additivity) for any $x, y \in G$ $\rho(xy) \leq \rho(x) + \rho(y)$. 
\end{enumerate}
\end{enumerate} 
\label{dnormedgroup}
\end{definition}

\begin{proposition}
Any normed group $(G, \rho)$ can be seen as a metric space, with any of the distances 
$$d_{L}(x, y) \, = \, \rho(x^{-1} y)  \quad , \quad d_{R}(x, y) \, = \, \rho(x y^{-1}) \quad .$$
The function $\displaystyle d_{L}$ is {\em left}-invariant, i.e.  for any $x, y, z \in G$ we have 
$\displaystyle d_{L}(zx, zy) \, = \, d_{L}(x,y)$. Similarly $\displaystyle d_{R}$ is {\em right}-invariant, 
that is for any $x, y, z \in G$ we have 
$\displaystyle d_{R}(xz, yz) \, = \, d_{R}(x,y)$. 
\end{proposition}

\begin{proof}
It suffices to give the proof for the distance $\displaystyle d_{L}$. 
Indeed, the first axiom of a distance is a consequence of the first axiom of a norm, the symmetry axiom for distances is a consequence of the symmetry axiom of the norm and the triangle inequality  comes from the group identity 
$$ x^{-1} z \, = \, \left( x^{-1} y \right) \left( y^{-1} z\right)$$
(which itself is a consequence of the associativity of the group operation and of the existence of inverse) 
and from the sub-additivity of the norm. 
The left-invariance of $\displaystyle d_{L}$ comes from the group identity $\displaystyle 
\left(zx\right)^{-1} \left(zy\right) = x^{-1} y$.  
\end{proof}

Groupoids are generalization of groups. A groupoid can be seen as a small category such that 
any arrow is invertible. Alternatively, if we look at the set of arrows of such a category, it is 
  a set with a partially defined binary operation and a unary operation (the inverse function), which satisfy  
several properties. A norm is then  a function defined on the set of arrows of a groupoid, with properties similar with the ones of a norm over a group.  This is the definition which we give further.

\begin{definition} 
A normed groupoid $(G, \rho)$ is a pair formed by: 
\begin{enumerate}
\item[-]  a groupoid  $G$, which is a set with two operations $\displaystyle inv: G \rightarrow 
G$,  $\displaystyle m: G^{(2)} \subset G \times
G \rightarrow G$, which satisfy a number of properties. With the notations 
$\displaystyle inv(a) = a^{-1}$, $\displaystyle m(a,b) = ab$, these properties
are: for any $a,b,c \in G$ 
\begin{enumerate}
\item[(i)] if $\displaystyle (a,b) \in G^{(2)}$ and 
$\displaystyle (b,c) \in G^{(2)}$ then $\displaystyle (a,bc) \in G^{(2)}$ and 
$\displaystyle (ab, c) \in G^{(2)}$ and we have $a(bc) = (ab)c$, 
\item[(ii)] $\displaystyle (a,a^{-1}) \in G^{(2)}$ and 
$\displaystyle (a^{-1},a) \in G^{(2)}$, 
\item[(iii)] if $\displaystyle (a,b) \in G^{(2)}$ then $\displaystyle a b b^{-1} = a$ and 
$\displaystyle a^{-1} a b = b$. 
\end{enumerate}
The set $X= Ob(G)$ is formed by all products 
$\displaystyle a^{-1} a$, $a \in G$. For any $a \in G$ we let $\alpha(a) = 
a^{-1} a$ and $\omega(a) = a a^{-1}$.
\item[-]   a norm function $d : G \rightarrow [0, +\infty)$ which satisfies the following axioms: 
\begin{enumerate}
\item[(i)] $d(g) = 0$ if and only if $g \in Ob(G)$, 
\item[(ii)] (symmetry) for any $g \in G$,  $d(g^{-1}) \, = \, d(g)$, 
\item[(iii)] (sub-additivity) for any $\displaystyle (g,h) \in G^{(2)}$,   
$d(gh) \, \leq \, d(g) + d(h)$,  
\end{enumerate}
\end{enumerate} 
\label{dnormedgroupoid}
\end{definition}

If $Ob(G)$ is a singleton then $G$ is just a group and the previous definition corresponds exactly to the definition \ref{dnormedgroup} of a normed group. As in the case of normed groups,  normed groupoids induce metric spaces too. 

\begin{proposition}
Let $(G,d)$ be  a normed groupoid  and  $x \in Ob(G)$.  Then the space $\displaystyle (\alpha^{-1}(x), 
d_{x})$ is a metric space, with the distance $\displaystyle d_{x}$ defined by: for any $g, h \in G$ with 
$\alpha(g) = \alpha(h) = x$ we have $\displaystyle d_{x}(g,h) \, = \, d( g h^{-1})$. 

Therefore a normed groupoid can be seen as a disjoint union of metric spaces 
\begin{equation}
G \, = \, \bigcup_{x \in Ob(G)} \alpha^{-1}(x) \quad , 
\label{disu}
\end{equation}
with the property that right translations in the groupoid are isometries, that is:  for any $u \in G$ the transformation 
$$\displaystyle R_{u} : \alpha^{-1}\left(\omega(u)\right) \rightarrow \alpha^{-1}\left(\alpha(u)\right) \quad , \quad R_{u} (g) \, = \, gu$$
has the property for any $\displaystyle g, h \in \alpha^{-1}\left(\omega(u)\right)$  
$$d_{\omega(u)} (g,h) \, = \, d_{\alpha(u)} ( R_{u}(g) , R_{u}(h)) \quad .$$
\label{pgroupoid}
\end{proposition}

\begin{proof}
We begin by noticing that if $\alpha(g) = \alpha(h)$ then $\displaystyle (g, h^{-1}) \in G^{(2)}$, therefore 
the expression $\displaystyle g h^{-1}$ makes sense. The rest of the proof of the first part of the proposition is identical with the proof of the previous proposition. 

For the proof of the second part of the proposition remark first that $\displaystyle R_{u}$ is well defined and that 
$$R_{u}(g) \left( R_{u}(h)\right)^{-1} \, = \, g h^{-1} \quad .$$
Then we have:
$$d_{\alpha(u)} ( R_{u}(g) , R_{u}(h))  \, = \, d\left( R_{u}(g) \left( R_{u}(h)\right)^{-1}\right) \, = \, $$ 
$$\, = \, d(g h^{-1}) \, = \, d_{\omega(u)} (g,h) \quad .$$
\end{proof}

Therefore normed groupoids provide examples of (disjoint unions of) metric spaces. Are there metric 
spaces more general than these? No, in fact we have the following. 

\begin{proposition}
Any metric space can be constructed from a normed groupoid, as in proposition \ref{pgroupoid}. Precisely, let $(X,d)$ be a metric space and consider the {\em trivial groupoid} $G = X \times X$ with 
multiplication 
$$(x,y) (y,z) \, = \, (x,z)$$
and inverse $\displaystyle (x,y)^{-1} = (y,x)$. Then $(G,d)$ is a normed groupoid and moreover 
any component of the decomposition (\ref{disu}) of $G$  is isometric with $(X,d)$. 

Conversely, if $G = X \times X$ is the trivial groupoid associated to the set $X$ and $d$ is a norm on 
$G$ then $(X,d)$ is a metric space.
\label{enoughgen}
\end{proposition}

\begin{proof}
We begin by noticing that $\alpha(x,y) = (y, y)$, $\omega(x,y) = (x,x)$, therefore 
$Ob(G) \, = \, \left\{ (x,x) \mbox{ : } x \in X\right\}$ can be identified with $X$ by the bijection 
$(x,x) \mapsto x$.  Moreover, for any $x  \in X$ we have 
$$\alpha^{-1}((x,x)) \, = \, X \times \left\{ x\right\} \quad . $$

Because $d: X\times X \rightarrow [0,+\infty)$ and $G = X \times X$ it follows that 
$d: G \rightarrow [0, +\infty)$. We have to check the properties of a norm over a groupoid. 
But these are straightforward. The statement (i) ($d(x,y) = 0$ if and only if $(x,y) \in Ob(G)$) is equivalent with $d(x,y) = 0$ if and only if $x = y$. The symmetry condition (ii) is just the symmetry of the distance: 
$d(x,y) = d(y,x)$. Finally the sub-additivity of $d$ seen as defined on the groupoid $G$ is equivalent with 
the triangle inequality: 
$$d((x,y)(y,z)) \, = \, d(x,z) \, \leq \, d(x,y) + d(y,z) \quad .$$
In conclusion $(G,d)$ is a normed groupoid if and only if $(X,d)$ is a metric space. 

For any $x  \in X$ the distance $\displaystyle d_{(x,x)}$ on the space 
$\displaystyle \alpha^{-1}((x,x)) $ has the expression: 
$$d_{(x,x)}( (u,x) , (v, x) ) \, = \, d((u,x) (v,x)^{-1}) \, = \, d((u,x)(x,v)) \, = \, d(u,v)$$
therefore the metric space $\displaystyle (\alpha^{-1}((x,x)), d_{(x,x)}) $ is isometric with 
$(X,d)$ by the isometry $(u,x) \mapsto u$, for any $u \in X$. 
\end{proof}

In conclusion normed groups give particular examples of metric spaces and metric spaces are particular examples of normed groupoids.  For this reason normed groups make good examples of metric spaces.  
It is also interesting to extend the theory of metric spaces to normed groupoids (other than trivial normed groupoids). This is done in \cite{groupoids}.

 \subsection{Gromov-Hausdorff distance}

%%%%%%%%%%%%%%%%%%%%%%%

For this subject see  \cite{burago} (Section 7.4),  \cite{gromov} (Chapter 3) and \cite{gromovgr}. We start 
the presentation  by a discussion about maps and microscopes.

Imagine that  the metric space $(X,d)$ represents some part of the world, like the collection of 
the cities in a country. We also need a injective function $Name  : X \rightarrow A$, 
which associated to any $x \in X$ the object $Name(x)$ which represents the name of the 
place $x$ (the set $A$ is a collection of names). (It  seems that  the function $Name$ is not 
really necessary for this process. Indeed, in this abstract mathematical 
description we use 
statements like "to $x \in X$ we associate $y \in Y$", so the letters $x, y$ are just 
generic names. In conclusion, in the following we may take $Name(x) = x$ without altering 
the discussion). 

The distance 
between two places called $Name(x)$ and $Name(y)$ is equal to $d(x,y)$.  
Suppose that we want 
to mathematically describe what is a map of the collection $(X,d, Name)$ in the 
metric space $(Y,d')$, at the scale $\varepsilon > 0$.  
For example $(Y, d')$ might represent a printed map of the region $(X,d, Name)$. 
For the moment 
we take $\varepsilon = 1$, meaning that we want to make a map of $(X,d, Name)$, 
at the scale 1:1, 
in $(Y,d')$.  

We might say that such a map of $(X,d, Name)$ in $(Y,d')$ is in fact a relation 
$\rho \subset X \times Y$. To the place named $Name(x)$ is associated the set of 
points $\left\{ y \in Y \mbox{ : } (x,y) \in \rho \right\}$. We then decorate our map $\rho$ with names 
by defining {\em a relation} $Name' \subset Y \times A$, given like this:  $(x,y) \in \rho$ if and only if 
$(y, Name(x)) \in Name'$. At this point we would like that the map  $\rho$  preserves 
the distances up to a precision $\mu$.

Let us simplify our notations concerning relations. For any relation  $\rho  \subset X \times Y$  we shall write $\rho(x) = y$ if $(x,y) \in \rho$. Therefore we may have $\rho(x) = y$ and $\rho(x) = y'$ with 
$y \not = y'$, if $(x,y) \in f$ and $(x,y') \in f$. 

The domain of the relation $\rho$ is the set  $dom \ \rho \,  \subset X$ such that for any $x \in \, dom \ \rho$ 
there is $y \in Y$ with $\rho(x) = y$. The image of $\rho$ is the set of $im \ \rho \  \subset Y$ such that for any $y \in \, im \ \rho$ there is $x \in X$ with $\rho(x) =  y$.  
By convention, when we write that a statement $R(f(x), f(y), ...)$ is true, we mean that $R(x',y', ...)$ is 
true for any choice of $x', y', ...$, such that $(x,x'), (y,y'), ... \in f$. 

When we make a map $\rho$ we are not really measuring the distances between all points in $X$, then 
consider a bijection from $X$ to $Y$. What we do is that first  we take, for a number $\mu > 0$, a collection $M \subset X$ of points in $X$ which is $\mu$-dense in $(X,d)$. 

\begin{definition}
A subset $M \subset X$ of a metric space $(X,d)$ is $\mu$-dense in $X$ if for any 
$u \in X$ there is $x \in M$ such that $d(x,u) \leq \mu$. 
\end{definition}

After measuring (or using other means to deduce) the distances $d(x', x")$ between 
all pairs of 
points in $M$ (we may have several values for the distance $d(x',x")$), 
we try to represent the collection 
of these distances  in $(Y,d')$.  Therefore we pick a subset $M' \subset Y$, which is  
$\mu$-dense in 
$(Y,d')$, maybe in order to spare the material of this expensive 1:1 map. 
Then we associate to any 
$x \in M$ one or several points $y \in M'$ such that for any two points 
$\displaystyle x_{1}, x_{2} \in M$ and for any choice of points $y_{1}, y_{2} \in M'$, in 
correspondence with  $\displaystyle x_{1}, x_{2}$ respectively, the distances 
$\displaystyle d(x_{1}, x_{2})$ and $\displaystyle d'(y_{1}, y_{2})$ differ by 
$\mu$ at most. The association to any point $x \in M$ of a point $y \in M'$ is the 
relation $\rho$, with 
domain $M$ and image $M'$. 

The infimum of all $\mu >0$ for which such a  map $\rho$ is possible represents the 
greatest precision of 
making a map of $(X,d)$ in $(Y, d')$. 

This infimum is in general not equal to zero.  We may treat symmetrically the metric spaces $(X,d)$ and 
$(Y,d')$ and ask for the infimum of all $\mu$ such that $(X,d)$ admits a map in $(Y,d')$ with precision 
$\mu$ and $(Y,d')$ admits a map in $(X,d)$ with precision $\mu$. This $\mu$ is called the Gromov-Hausdorff distance between the metric spaces $(X,d)$ and $(Y,d')$. This distance can be also infinite if 
for any $\mu$ we cannot have a map $\rho$ associated.

We shall use  also the following convenient notation: by $\mathcal{O}(\varepsilon)$ we mean a positive function such that $\displaystyle \lim_{\varepsilon \rightarrow 0} \mathcal{O}(\varepsilon) = 0$. 

The  definition of the  Gromov-Hausdorff distance for pointed metric spaces is the following. 

\begin{definition}
Let $\displaystyle (X_{i},  d_{i}, x_{i})$, $i=1,2$, be a pair of locally compact pointed metric spaces and $\mu > 0$. 
We shall say that $\mu$ is admissible if  there is a relation $\displaystyle \rho \subset X_{1} \times X_{2}$ such that 
\begin{enumerate}
\itemsep-3pt
\item[1.] $dom \ \rho$ is $\mu$-dense in $\displaystyle X_{1}$, 
\item[2.] $im \ \rho$ is $\mu$-dense in $\displaystyle X_{2}$, 
\item[3.] $\displaystyle (x_{1}, x_{2}) \in \rho$, 
\item[4.] for all $x,y \in \ dom \ \rho$ we have 
\begin{equation}
\mid d_{2}(\rho(x), \rho(y)) - d_{1}(x,y) \mid \ \leq \  \mu
\label{photoquasi}
\end{equation}
\end{enumerate}
The Gromov-Hausdorff distance  between $\displaystyle (X_{1}, x_{1}, d_{1})$ and $\displaystyle  (X_{2}, x_{2}, d_{2})$  is   the infimum of admissible numbers $\mu$. 
\label{defgh}
\end{definition}

As introduced in definition \ref{defgh}, the Gromov-Hausdorff  (GH) distance is not a true 
distance, because the GH distance between two isometric pointed metric spaces   
is equal to zero. In fact the  
 GH distance induces a distance on isometry classes of pointed metric spaces 
 (which are not far apart). 
(The isometry class  $\displaystyle [X,d_{X}, x]$  of the pointed metric space $\displaystyle (X,d_{X}, x)$,  is the class of spaces $\displaystyle (Y,d_{Y}, y)$ such 
that it exists an isometry  $f:X \rightarrow Y$ with the property $f(x)=y$. )

Indeed, if two pointed metric spaces are isometric then the Gromov-Hausdorff distance  
equals $0$. 
The converse is also true  in the class of compact (pointed) metric spaces  
\cite{gromov} (Proposition 3.6).

Moreover, if two  of the isometry classes  $[X,d_{X}, x]$, $[Y,d_{Y}, y]$, $[Z,d_{Z}, z]$ have 
(representants with) diameter at most equal to 3, then the triangle inequality is true. We  shall 
use this distance and the induced  convergence for isometry classes of the form $[X,d_{X}, x]$, 
with $diam \ X \ \leq 5/2$. 

\subsection{Metric profiles. Metric tangent space}

We shall denote by $CMS$ the set of isometry classes of pointed compact metric spaces. The distance on this set is the Gromov distance between (isometry classes of) pointed metric spaces and the topology 
is induced by this distance. 

To any locally compact metric  space we can associate a metric profile \cite{buliga3,buliga4}. 

\begin{definition}
\label{dmprof}
The metric profile associated to the locally metric space $(M,d)$ is  the assignment 
(for small enough $\varepsilon > 0$) 
\[(\varepsilon > 0 , \ x \in M) \ \mapsto \  \mathbb{P}^{m}(\varepsilon, x) = \left[\bar{B}(x,1), 
\frac{1}{\varepsilon} d, x\right] \in CMS\]
\end{definition}

We can define a notion of metric profile regardless to any distance. 

\begin{definition}
\label{dprofile}
A metric profile is a curve $\mathbb{P}:[0,a] \rightarrow CMS$ such that
\begin{enumerate}
\itemsep-3pt
\item[(a)] it is continuous at $0$,
\item[(b)]for any $b \in [0,a]$ and  $\varepsilon \in (0,1]$ we have 
\[d_{GH} (\mathbb{P}(\varepsilon b), \mathbb{P}^{m}_{d_{b}}(\varepsilon,x_{b}))  \  = \ O(\varepsilon)\]
\end{enumerate}
The function $\mathcal{O}(\varepsilon)$ may change with  $b$.
We used the notations
\[
\mathbb{P}(b) = [\bar{B}(x,1) ,d_{b}, x_{b}] \quad \mbox{  and } \quad 
\mathbb{P}^{m}_{d_{b}}(\varepsilon,x) = \left[\bar{B}(x,1),\frac{1}{\varepsilon}d_{b}, x_{b}\right] 
\]
The metric profile is nice if 
\[d_{GH}\left (\mathbb{P}(\varepsilon b), \mathbb{P}^{m}_{d_{b}}(\varepsilon,x)\right) =O(b \varepsilon) \]
\end{definition}

Imagine that $1/b$ represents the magnification on the scale of a microscope. We use the microscope to study a specimen. For each $b > 0$ the information that we get is the table of distances of the pointed metric space $\displaystyle (\bar{B}(x,1) ,d_{b}, x_{b})$. 

How can we know, just from the information given by the microscope, that the string of "images" that we have corresponds to a real specimen? The answer is that a reasonable check is the relation from point (b) of the definition of metric profiles \ref{dprofile}. 

Really, this point says that starting from any magnification $1/b$,  if we further select the ball 
$\displaystyle \bar{B}(x, \varepsilon)$ in the snapshot  $\displaystyle (\bar{B}(x,1) ,d_{b}, x_{b})$, then 
the metric space $\displaystyle (\bar{B}(x,1) ,\frac{1}{\varepsilon} d_{b}, x_{b})$ looks approximately the same as the snapshot $\displaystyle (\bar{B}(x,1) ,d_{b\varepsilon}, x_{b})$. That is: further magnification by $\varepsilon$ of 
the snapshot (taken with magnification) $b$ is roughly the  same as the snapshot $b \varepsilon$. This 
is of course true in a neighbourhood of the base point $\displaystyle x_{b}$. 

The point (a) from the Definition \ref{dprofile}Ê has no other justification than Proposition \ref{propmetcone} in next 
subsection. 

We rewrite  definition \ref{defgh} with more details, in order to clearly understand what is a metric profile.  For any $b \in (0,a]$ and for any $\mu > 0$ there is $\varepsilon(\mu, b) \in (0,1)$ such that for any $\varepsilon \in 
(0,\varepsilon(\mu,b))$ there exists a relation $\displaystyle \rho = \rho_{\varepsilon, b} \subset \bar{B}_{d_{b}}(x_{b}, \varepsilon) \times \bar{B}_{d_{b \varepsilon}}(x_{b \varepsilon}, 1)$ such that
\begin{enumerate}
\itemsep-3pt
\item[1.] $\displaystyle dom \ \rho_{\varepsilon, b}$ is $\mu$-dense in $\displaystyle \bar{B}_{d_{b}}(x_{b}, \varepsilon)$, 
\item[2.] $\displaystyle im \ \rho_{\varepsilon, b}$ is $\mu$-dense in $\displaystyle \bar{B}_{d_{b \varepsilon}}(x_{b \varepsilon}, 1)$, 
\item[3.] $\displaystyle (x_{b}, x_{b\varepsilon}) \in \rho_{\varepsilon, b}$, 
\item[4.] for all $\displaystyle x,y \in \ dom \ \rho_{\varepsilon, b}$ we have
\begin{equation}
\label{disest}
\left|\frac{1}{\varepsilon} d_{b}(x,y) 
- d_{b \varepsilon}\left(\rho_{\varepsilon, b}(x), \rho_{\varepsilon, b}(y)\right)\right|
\, \leq \,  \mu
\end{equation}
\end{enumerate}

In the microscope interpretation, if $\displaystyle (x,u) \in \rho_{\varepsilon, b}$ means that $x$ and $u$ represent the same "real" point in the specimen. 

Therefore a metric profile gives two types of information:
\begin{itemize}
\itemsep-3pt
\item
a distance estimate like (\ref{disest}) from point 4,  
\item
an "approximate shape" estimate, like in the points 1--3, where we see that two sets, namely the balls   $\displaystyle \bar{B}_{d_{b}}(x_{b}, \varepsilon)$ and  $\displaystyle \bar{B}_{d_{b \varepsilon}}(x_{b \varepsilon}, 1)$, are approximately isometric. 
\end{itemize}

The simplest metric profile is one with $\displaystyle (\bar{B}(x_{b}, 1), d_{b}, x_{b}) = (X, d_{b}, x)$. 
In this case we see that $\displaystyle \rho_{\varepsilon, b}$ 
is approximately an $\varepsilon$ dilatation with base point $x$.  

This observation leads us to a particular class of (pointed) metric spaces, 
namely the  metric cones. 

\begin{definition}
\label{defmetcone}
A metric cone $(X,d, x)$ is a locally compact metric space $(X,d)$, with a marked point $x \in X$ such 
that for any $a,b \in (0,1]$ we have 
\[\displaystyle \mathbb{P}^{m}(a,x)  =  \mathbb{P}^{m}(b,x)\] 
\end{definition}

Metric cones have dilatations. By this we mean the following

\begin{definition}
Let $(X,d, x)$ be a metric cone. For any $\varepsilon \in (0,1]$  a dilatation is a function $\displaystyle \delta^{x}_{\varepsilon}: \bar{B}(x,1) \rightarrow \bar{B}(x,\varepsilon)$ such that 
\begin{itemize}
\itemsep-3pt
\item
$\displaystyle \delta^{x}_{\varepsilon}(x) = x$, 
\item
for any $u,v \in X$ we have 
\[d\left(\delta^{x}_{\varepsilon}(u), \delta^{x}_{\varepsilon}(v)\right) =\varepsilon \, d(u,v) \]
\end{itemize}
\end{definition}

The existence of dilatations for metric cones comes from the definition \ref{defmetcone}. 
Indeed, dilatations are just  isometries from $\displaystyle (\bar{B}(x,1), d, x)$ to $ (\bar{B}, \frac{1}{a}d, x)$. 

Metric cones are good candidates for being tangent spaces in the metric sense. 

\begin{definition}
\label{defmetspace}
A (locally compact) metric space $(M,d)$ admits a (metric) tangent space in $x \in M$ if the associated metric profile  
$\varepsilon \mapsto \mathbb{P}^{m}(\varepsilon, x)$ (as in definition \ref{dmprof})  admits a prolongation by continuity in 
$\varepsilon = 0$, i.e if the following limit exists: 
\begin{equation}
\label{limmetspace}
[T_{x}M,d^{x}, x]  =  \lim_{\varepsilon \rightarrow 0} \mathbb{P}^{m}(\varepsilon, x)
\end{equation}
\end{definition}

The connection between metric cones, tangent spaces and metric profiles in the abstract sense is made by the following proposition. 

\begin{proposition}
\label{propmetcone}
The associated metric profile $\varepsilon \mapsto \mathbb{P}^{m}(\varepsilon, x)$ of a metric space $(M,d)$ for  a fixed $x \in M$ is a metric profile in the sense of the definition \ref{dprofile} if and only if the space 
$(M,d)$ admits a tangent space in $x$. 
In such a case the tangent space is a metric cone. 
\end{proposition}

\begin{proof} 
A tangent space  $[V,d_{v}, v]$ exists if and only if we have the limit from the relation (\ref{limmetspace}). 
In this case there exists a prolongation by continuity to $\varepsilon = 0$  of the  metric profile 
$\mathbb{P}^{m}(\cdot , x)$. 
The prolongation is a metric profile in the sense of definition \ref{dprofile}. 
Indeed, we have still to check the property (b). But this is trivial, because for any $\varepsilon, b >0$,  sufficiently small, we have 
\[
\mathbb{P}^{m}(\varepsilon b, x) =  \mathbb{P}^{m}_{d_{b}}(\varepsilon,x)
\]
where  $d_{b} = (1/b) d$ and 
$\mathbb{P}^{m}_{d_{b}}(\varepsilon,x) = [\bar{B}(x,1),\frac{1}{\varepsilon}d_{b}, x]$.

Finally, let us prove that the tangent space is a metric cone. For any $a \in (0,1]$ we have
\[
\left[\bar{B}(x,1), \frac{1}{a}d^{x}, x\right] = \lim_{\varepsilon \rightarrow 0} \mathbb{P}^{m}(a \varepsilon, x) \]
Therefore 
\[
\left[\bar{B}(x,1), \frac{1}{a}d^{x}, x\right]  = [T_{x}M,d^{x}, x] 
 \tag*{\qed}
\] 
\renewcommand{\qed}{}
\end{proof}

\subsection{Length in metric spaces}

For a detailed introduction into the subject see for example \cite{amb}, chapter
1. 

\begin{definition}
The {\bf (upper) dilatation of a map} $f: X \rightarrow Y$ between metric spaces,  
in a point $u \in Y$ is 
$$ Lip(f)(u) = \limsup_{\varepsilon \rightarrow 0} \  
\sup  \left\{ 
\frac{d_{Y}(f(v), f(w))}{d_{X}(v,w)} \ : \ v \not = w \ , \ v,w \in B(u,\varepsilon)
 \right\}$$
\end{definition}
In the particular case of a   derivable function 
$f: \mathbb{R} \rightarrow \mathbb{R}^{n}$ the upper dilatation is  
$\displaystyle Lip(f)(t)  =  \|\dot{f}(t)\|$. 

 A function  $f:(X, d) \rightarrow (Y, d')$  is Lipschitz if there is a positive 
 constant $C$ such that for any $x,y \in X$ we have 
 $\displaystyle d'(f(x),f(y)) \leq C \, d(x,y)$. The number $Lip(f)$ is the smallest 
 such positive constant. Then  for any $x \in X$ we have the obvious relation  
$\displaystyle Lip(f)(x) \ \leq \ Lip(f)$.

A curve is a continuous function $c: [a,b] \rightarrow X$. The image of a curve is 
called path. Length measures paths. Therefore length does not depends on the 
reparameterization of the path and it is additive with respect to concatenation of paths.

\begin{definition}
In a metric space $(X,d)$ there are several ways to define  the length: 
\begin{enumerate}
\item[(a)] The {\bf length 
of a  curve with $L^{1}$ upper dilatation} $c: [a,b] \rightarrow X$ is 
$$L(f) = \int_{a}^{b} Lip(c)(t) \mbox{ d}t$$
\item[(b)] The {\bf variation of a curve}  $c: [a,b] \rightarrow X$ is the quantity 
$ Var(c)  =$  
$$  =  \sup \left\{ \sum_{i=0}^{n} d(c(t_{i}), 
c(t_{i+1})) \ \mbox{ : }  a = t_{0} <  t_{1} < ... < t_{n} < t_{n+1} =  b
\right\}$$ 
\item[(c)] The {\bf length of the path} $A = c([a,b])$ is  the
one-dimensional Hausdorff measure of the path.:  
$$l(A) \ = \ \lim_{\delta \rightarrow 0}  
 \inf \left\{ \sum_{i \in I} diam \ E_{i}  \mbox{ : } diam \ E_{i} 
< \delta \ , \ \ A \subset \bigcup_{i \in I} E_{i} \right\} $$
\end{enumerate}
\label{deflenght}
\end{definition}

The definitions are not equivalent. For Lipschitz curves the first  two 
definitions agree. For simple Lipschitz  curves all definitions agree.

\begin{theorem}
For each Lipschitz curve $c: [a,b] \rightarrow X$, we have 
$\displaystyle L(c) \ = \ Var(c) \ \geq \ \mathcal{H}^{1}(c([a,b]))$.

 If $c$ is moreover injective then  $\displaystyle \mathcal{H}^{1}(c([a,b])) \  = 
 \ Var(f)$. 
\label{t411amb}
\end{theorem}

An important tool used in the proof  of the previous theorem is 
the geometrically obvious, but not straightforward to prove in this generality, 
Reparametrisation Theorem. 

\begin{theorem}
Any Lipschitz curve  admits a reparametrisation $c:
[a,b] \rightarrow A$ such that $Lip(c)(t) = 1$ for almost any $t \in [a,b]$. 
\label{tp}
\end{theorem}

\begin{definition}
We shall denote by $l_{d}$ the {\bf length functional induced by the distance} $d$, 
defined only on the family of Lipschitz curves. 
  If the metric space $(X,d)$ is connected by Lipschitz curves, then the length 
induces a new distance $d_{l}$, given by: 
$$d_{l}(x,y) \  = \ \inf \ \left\{ l_{d}(c([a,b])) \mbox{ : } 
c: [a,b] \rightarrow X \ \mbox{ Lipschitz } , \right.$$
$$\left. \ c(a)=x \ , \ c(b) = y \right\}$$

A {\bf length metric space} is a metric space $(X,d)$, connected by Lipschitz curves, 
 such that $d  = d_{l}$. 
\label{dpath}
\end{definition}

From theorem \ref{t411amb} we deduce that Lipschitz curves in complete 
length metric spaces are absolutely continuous. 
Indeed, here is the definition of an absolutely continuous curve 
(definition 1.1.1, chapter 1,   \cite{amb}). 
 
 \begin{definition}
 Let $(X,d)$ be a complete metric space. A curve $c:(a,b)\rightarrow X$ is {\bf absolutely 
 continuous} if there exists $m\in L^{1}((a,b))$ such that for any $a<s\leq t<b$ we have 
 $$d(c(s),c(t)) \leq \int_{s}^{t} m(r) \mbox{ d}r   .$$
 Such a function $m$ is called a {\bf upper gradient} of the curve $c$. 
 \label{defac}
 \end{definition}
 
 According to theorem \ref{t411amb}, for a Lipschitz curve $c:[a,b]\rightarrow X$ in a 
 complete length metric space such a function 
 $m\in L^{1}((a,b))$  is the upper dilatation  $Lip(c)$. 
More can be said about the expression of the upper dilatation. We need first to introduce the notion of 
metric derivative of a Lipschitz curve. 

\begin{definition}
A curve $c:(a,b)\rightarrow X$ is {\bf metrically derivable} in $t\in(a,b)$ if the limit 
$$md(c)(t) = \lim_{s\rightarrow t} \frac{d(c(s),c(t))}{\mid s-t \mid}$$
exists and it is finite. In this case $md(c)(t)$ is called the {\bf metric
derivative} of $c$ in $t$. 
\label{defmd}
\end{definition}

For the proof of the following theorem see \cite{amb}, theorem 1.1.2, chapter 1. 

\begin{theorem}
Let $(X,d)$ be a complete metric space and $c:(a,b)\rightarrow X$ be an absolutely continuous curve. 
Then $c$ is metrically  derivable for $\mathcal{L}^{1}$-a.e. $t\in(a,b)$. Moreover the function $md(c)$ belongs to $L^{1}((a,b))$ and it is minimal in the following sense: $md(c)(t)\leq m(t)$  for  $\mathcal{L}^{1}$-a.e. $t\in(a,b)$, for each upper gradient $m$ of the curve $c$. 
\label{tupper}
\end{theorem}

\section{Metric spaces with dilations}

We shall use here a slightly particular version of dilatation structures. 
For the general definition of a dilatation structure see \cite{buligadil1} (the general definition 
applies for dilatation structures over ultrametric spaces as well).

\begin{definition}
Let $(X,d)$ be a complete metric space such that for any $x  \in X$ the 
closed ball $\bar{B}(x,3)$ is compact. A {\bf dilatation structure} $(X,d, \delta)$ 
over $(X,d)$ is the assignment to any $x \in X$  and $\varepsilon \in (0,+\infty)$ 
of a invertible homeomorphism, defined as: if 
$\displaystyle   \varepsilon \in (0, 1]$ then  $\displaystyle 
 \delta^{x}_{\varepsilon} : U(x)
\rightarrow V_{\varepsilon}(x)$, else 
$\displaystyle  \delta^{x}_{\varepsilon} : 
W_{\varepsilon}(x) \rightarrow U(x)$, such that the following axioms are satisfied: 
\begin{enumerate}
\item[{\bf A0.}]  there are  numbers  $1<A<B$ such that for any $x \in X$  and any 
$\varepsilon \in (0,1)$ we have 
  the following string of inclusions:
$$ B_{d}(x, \varepsilon) \subset \delta^{x}_{\varepsilon}  B_{d}(x, A) 
\subset V_{\varepsilon}(x) \subset 
W_{\varepsilon^{-1}}(x) \subset \delta_{\varepsilon}^{x}  B_{d}(x, B) $$
Moreover for any compact set $K \subset X$ there are $R=R(K) > 0$ and 
$\displaystyle \varepsilon_{0}= \varepsilon(K) \in (0,1)$  such that  
for all $\displaystyle u,v \in \bar{B}_{d}(x,R)$ and all 
$\displaystyle \varepsilon  \in (0,\varepsilon_{0})$,  we have 
$$\delta_{\varepsilon}^{x} v \in W_{\varepsilon^{-1}}( \delta^{x}_{\varepsilon}u) \ .$$

\item[{\bf A1.}]  We  have 
$\displaystyle  \delta^{x}_{\varepsilon} x = x $ for any point $x$. 
We also have $\displaystyle \delta^{x}_{1} = id$ for any $x \in X$. 
Let us define the topological space
$$ dom \, \delta = \left\{ (\varepsilon, x, y) \in (0,+\infty) \times X 
\times X \mbox{ :  if } \varepsilon \leq 1 \mbox{ then } y 
\in U(x) \,
\, , 
\right.$$ 
$$\left. \mbox{  else } y \in W_{\varepsilon}(x) \right\} $$ 
with the topology inherited from $(0,+\infty) \times X \times X$ endowed with
 the product topology. Consider also 
$\displaystyle Cl(dom \, \delta)$, 
the closure of 
$dom \, \delta$ in $\displaystyle [0,+\infty) \times X \times X$. The function $\displaystyle \delta : dom \, \delta 
\rightarrow  X$ defined by $\displaystyle \delta (\varepsilon,  x, y)  = 
\delta^{x}_{\varepsilon} y$ is continuous. Moreover, it can be continuously 
extended to the set $\displaystyle Cl(dom \, \delta)$ and we have 
$$\lim_{\varepsilon\rightarrow 0} \delta_{\varepsilon}^{x} y \, = \, x  $$

\item[{\bf A2.}] For any  $x, \in X$, $\displaystyle \varepsilon, \mu \in (0,+\infty)$
 and $\displaystyle u \in U(x)$   we have the equality: 
$$ \delta_{\varepsilon}^{x} \delta_{\mu}^{x} u  = \delta_{\varepsilon \mu}^{x} u $$ 
whenever one of the sides are well defined.

\item[{\bf A3.}]  For any $x$ there is a distance  function $\displaystyle (u,v) \mapsto d^{x}(u,v)$, defined for any $u,v$ in the closed ball (in distance d) $\displaystyle 
\bar{B}(x,A)$, such that 
$$\lim_{\varepsilon \rightarrow 0} \quad \sup  \left\{  \mid 
\frac{1}{\varepsilon} d(\delta^{x}_{\varepsilon} u, \delta^{x}_{\varepsilon} v) \ - \ d^{x}(u,v) \mid \mbox{ :  } u,v \in \bar{B}_{d}(x,A)\right\} \ =  \ 0$$
uniformly with respect to $x$ in compact set. 

\end{enumerate}

The {\bf dilatation structure is strong} if it satisfies the following supplementary condition: 

\begin{enumerate}
\item[{\bf A4.}] Let us define 
$\displaystyle \Delta^{x}_{\varepsilon}(u,v) =
\delta_{\varepsilon^{-1}}^{\delta^{x}_{\varepsilon} u} \delta^{x}_{\varepsilon} v$. 
Then we have the limit 
$$\lim_{\varepsilon \rightarrow 0}  \Delta^{x}_{\varepsilon}(u,v) =  \Delta^{x}(u, v)  $$
uniformly with respect to $x, u, v$ in compact set. 
\end{enumerate}
\label{defweakstrong}
\end{definition}
 
We shall use many times from now the words "sufficiently close". This deserves 
a definition. 

\begin{definition}
Let  $(X,d, \delta)$ be a strong dilatation structure. We say that a property 
$\displaystyle \mathcal{P}(x_{1},x_{2},
x_{3}, ...)$ holds for $\displaystyle x_{1}, x_{2}, x_{3},
...$ {\bf sufficiently close} if for any compact, non empty set $K \subset X$, there
is a positive constant $C(K)> 0$ such that $\displaystyle \mathcal{P}(x_{1},x_{2},
x_{3}, ...)$ is true for any $\displaystyle x_{1},x_{2},
x_{3}, ... \in K$ with $\displaystyle d(x_{i}, x_{j}) \leq C(K)$.
\end{definition}

We shall look at dilatation structures from the metric point of view, 
by  using Gromov-Hausdorff distance and metric profiles. 
 
 We state the interpretation of the Axiom A3  as a theorem. 
But before a definition: we denote by $(\delta, \varepsilon)$ the distance on 
\[\bar{B}_{d^{x}}(x,1) = \left\{ y \in X \mbox{: } d^{x}(x,y) \leq 1 \right\}\] given by
\[(\delta, \varepsilon)(u,v) = \frac{1}{\varepsilon} d(\delta^{x}_{\varepsilon} u , \delta^{x}_{\varepsilon} v) \]

\begin{theorem}
\label{thcone}
Let $(X,d,\delta)$ be a dilatation structure. The following are consequences of the Axioms A0~-~A3 only: 
\begin{enumerate}
\item[(a)] 
for all $u,v \in X$ such that $\displaystyle d(x,u)\leq 1$ and $\displaystyle d(x,v) \leq 1$  and all $\mu \in (0,A)$ we have
\[d^{x}(u,v) = \frac{1}{\mu} d^{x}(\delta_{\mu}^{x} u , \delta^{x}_{\mu} v) \]
We shall say that $d^{x}$ has the cone property with respect to dilatations. 
\item[(b)] 
The curve $\displaystyle \varepsilon> 0 \mapsto \mathbb{P}^{x}(\varepsilon) = [\bar{B}_{d^{x}}(x,1), (\delta, \varepsilon), x]$ is a  metric profile.
\end{enumerate}
\end{theorem}

\begin{proof}
(a) For $\varepsilon, \mu \in (0,1)$ we have
\begin{align*}
\left| \frac{1}{\varepsilon \mu} d(\delta_{\varepsilon}^{x}\delta^{x}_{\mu} u, \delta_{\varepsilon}^{x}\delta^{x}_{\mu} v) - d^{x}(u,v) \right|
& \leq \,
\left| \frac{1}{\varepsilon \mu} d(\delta_{\varepsilon \mu}^{x} u, \delta_{\varepsilon}^{x}\delta^{x}_{\mu} u) -  
\frac{1}{\varepsilon \mu} d(\delta_{\varepsilon \mu}^{x} v, \delta_{\varepsilon}^{x}\delta^{x}_{\mu} v) \right|\\
& + \left|\frac{1}{\varepsilon \mu} d(\delta_{\varepsilon \mu}^{x}u, \delta_{\varepsilon \mu}^{x} v) - d^{x}(u,v)\right|
\end{align*}
Use now the Axioms A2 and A3 and pass to the limit with $\varepsilon \rightarrow 0$. This gives the desired equality. 

\medskip
(b)We have to prove that $\mathbb{P}^{x}$ is a metric profile. For this we have to compare two pointed metric spaces: 
\[ \left(\bar{B}_{d^{x}}(x,1), (\delta^{x}, \varepsilon \mu), x\right) \ \mbox{ and } \ \left(\bar{B}_{\frac{1}{\mu}(\delta^{x}, \varepsilon)}(x,1), \frac{1}{\mu}(\delta^{x}, \varepsilon), x \right) \]
Let $u \in X$ such that 
\[\frac{1}{\mu}(\delta^{x}, \varepsilon)(x, u) \leq 1 \]
This means that
\[\frac{1}{\varepsilon} d(\delta^{x}_{\varepsilon} x , \delta^{x}_{\varepsilon} u) \ \leq \ \mu \]
Further use the  Axioms  A1, A2 and the cone property proved before:
\[\frac{1}{\varepsilon} d^{x}(\delta^{x}_{\varepsilon} x, \delta^{x}_{\varepsilon} u) \ \leq \ (\mathcal{O}(\varepsilon) + 1) \mu\]
therefore, 
\[d^{x}(x,u)  \  \leq \ (\mathcal{O}(\varepsilon) + 1) \mu \]
It follows that for any $\displaystyle u \in \bar{B}_{\frac{1}{\mu}(\delta^{x}, \varepsilon)}(x,1)$ we can choose $\displaystyle w(u) \in \bar{B}_{d^{x}}(x,1)$ such that
\[\frac{1}{\mu} d^{x}(u, \delta^{x}_{\mu} w(u)) = \mathcal{O}(\varepsilon) \]
We want to prove that 
\[\mid \frac{1}{\mu}(\delta^{x}, \varepsilon) (u_{1}, u_{2}) - (\delta^{x}, \varepsilon \mu) (w(u_{1}), 
w(u_{2}) ) \mid \ \leq \ \mathcal{O}(\varepsilon \mu) + \frac{1}{\mu} \mathcal{O}(\varepsilon) + \mathcal{O}(\varepsilon) \ \]
This goes as follows: 
\begin{align*}
\lvert \frac{1}{\mu}(\delta^{x}, \varepsilon) (u_{1}, u_{2})  
&- (\delta^{x}, \varepsilon \mu) (w(u_{1}), w(u_{2}) ) \rvert
\,=\, \\ 
 & = \, \left| \frac{1}{\varepsilon \mu} d( \delta^{x}_{\varepsilon} u_{1}, \delta^{x}_{\varepsilon} u_{2}) -  \frac{1}{\varepsilon \mu} d(\delta^{x}_{\varepsilon \mu} w(u_{1}), \delta^{x}_{\varepsilon \mu} w(u_{2}))\right|  \\
&\leq 
\mathcal{O}(\varepsilon \mu) + \left|\frac{1}{\varepsilon \mu} d( \delta^{x}_{\varepsilon} u_{1}, \delta^{x}_{\varepsilon} u_{2}) - \frac{1}{\varepsilon \mu} d(\delta^{x}_{\varepsilon}\delta^{x}_{ \mu} w(u_{1}), \delta^{x}_{\varepsilon}\delta^{x}_{ \mu} w(u_{2}))\right| \\
&\leq 
\mathcal{O}(\varepsilon \mu) + \frac{1}{\mu} \mathcal{O}(\varepsilon) \, + \, 
\frac{1}{\mu} \mid d^{x} (u_{1}, u_{2}) - d^{x}(\delta^{x}_{\mu} w(u_{1}) , \delta^{x}_{\mu} w(u_{2}))\mid
\end{align*}
In order to obtain the last estimate we used twice  the Axiom A3. We proceed as follows:  
\begin{gather*}
\mathcal{O}(\varepsilon \mu) \ + \ \frac{1}{\mu} \mathcal{O}(\varepsilon) \ + \ 
\frac{1}{\mu} \mid d^{x} (u_{1}, u_{2}) - d^{x}(\delta^{x}_{\mu} w(u_{1}) , \delta^{x}_{\mu} w(u_{2})) \mid \, \leq \ \\ 
\leq \mathcal{O}(\varepsilon \mu) \ + \ \frac{1}{\mu} \mathcal{O}(\varepsilon) \ + \ \frac{1}{\mu} 
d^{x}(u_{1}, \delta^{x}_{\mu} w(u_{1})) \  + \ \frac{1}{\mu} 
d^{x}(u_{1}, \delta^{x}_{\mu} w(u_{2}))\\ 
\leq \  \mathcal{O}(\varepsilon \mu) + \frac{1}{\mu} \mathcal{O}(\varepsilon) + \mathcal{O}(\varepsilon) 
\end{gather*}
This shows that the property (b) of a metric profile is satisfied. 
The property (a)  is proved in  the Theorem \ref{tmit}. 
\end{proof}

The following theorem is related to  Mitchell \cite{mit} Theorem 1, concerning  sub-riemannian geometry. 

\begin{theorem}
\label{tmit}
In the hypothesis of theorem \ref{thcone}, we have  the following limit: 
\[\lim_{\varepsilon \rightarrow 0} \ \frac{1}{\varepsilon} \sup \left\{  \mid d(u,v) - d^{x}(u,v) \mid:  d(x,u) \leq \varepsilon, \ d(x,v) \leq \varepsilon \right\} = 0 \]
Therefore if $d^{x}$ is a true (i.e. nondegenerate) distance, then  $(X,d)$ admits a metric tangent space 
in $x$. 

Moreover, the metric profile $\displaystyle [\bar{B}_{d^{x}}(x,1), (\delta, \varepsilon), x]$ is  
almost nice, in the following sense. Let $c \in (0,1)$. Then  we have the inclusion
\[\delta^{x}_{\mu^{-1}}\left(\bar{B}_{\frac{1}{\mu}(\delta^{x}, \varepsilon)}(x,c)\right) \ \subset \ 
\bar{B}_{d^{x}}(x,1) \]
Moreover,  the following  Gromov-Hausdorff distance  is of order $\displaystyle \mathcal{O}(\varepsilon)$ 
for $\mu$ fixed (that is the modulus of convergence $\mathcal{O}(\varepsilon)$ does not depend on $\mu$): 
\[ \mu \ d_{GH}\left(  [\bar{B}_{d^{x}}(x,1) , (\delta^{x}, \varepsilon), x], \ [\delta^{x}_{\mu^{-1}}\left(\bar{B}_{\frac{1}{\mu}(\delta^{x}, \varepsilon)}(x,c)\right) , (\delta^{x}, \varepsilon \mu), x]\right) = 
\mathcal{O}(\varepsilon) \] 
For another Gromov-Hausdorff distance we have the estimate
\[d_{GH}\left( [\bar{B}_{\frac{1}{\mu}(\delta^{x}, \varepsilon)}(x,c), \frac{1}{\mu} (\delta^{x}, \varepsilon), x] \ , \ [\delta^{x}_{\mu^{-1}}\left(\bar{B}_{\frac{1}{\mu}(\delta^{x}, \varepsilon)}(x,c)\right) , (\delta^{x}, \varepsilon \mu), x]\right) = \mathcal{O}(\varepsilon \mu)\]
when $\varepsilon \in (0,\varepsilon(c))$. 
\end{theorem}

\begin{proof}
We start from the Axioms A0, A3 and we use the cone property. By A0,  for $\varepsilon \in (0,1)$ and $\displaystyle u,v \in  \bar{B}_{d}(x,\varepsilon)$ there exist $\displaystyle U,V  \in \bar{B}_{d}(x,A)$ such that \[u = \delta^{x}_{\varepsilon} U , v = \delta^{x}_{\varepsilon} V . \]
By the cone property we have 
\[\frac{1}{\varepsilon} \mid d(u,v) - d^{x}(u,v) \mid \,=\, \left|\frac{1}{\varepsilon} d(\delta^{x}_{\varepsilon} U,\delta^{x}_{\varepsilon} V) - d^{x}(U,V) \right| \]
By A2 we have 
\[  \left| \frac{1}{\varepsilon} d(\delta^{x}_{\varepsilon} U,\delta^{x}_{\varepsilon} V) - d^{x}(U,V) \right|\, \leq\, \mathcal{O}(\varepsilon)\]
This proves the first part of the theorem.

For the second part of the theorem take any 
$\displaystyle u \in \bar{B}_{\frac{1}{\mu}(\delta^{x}, \varepsilon)}(x,c)$. 
Then we have 
\[d^{x}(x,u) \, \leq \,  c \mu + \mathcal{O}(\varepsilon) \]
Then there exists   $\varepsilon(c) > 0$ such that 
for any $\varepsilon \in (0,\varepsilon(c))$ and $u$ in the mentioned ball we have
\[d^{x}(x,u) \,\leq \, \mu\]
In this case we can take directly $\displaystyle w(u) \,= \, \delta^{x}_{\mu^{-1}} u$ and simplify 
the  string of inequalities from the proof of  Theorem 
\ref{thcone}, point (b), to get eventually the three points 
from the second part of the theorem. 
\end{proof}

\section{Length metric spaces with dilations}
\label{seclds}

Consider $(X,d)$ a complete, locally compact metric space, and a triple 
 $(X,d,\delta)$  which satisfies  A0, A1, A2. Denote by $\displaystyle Lip([0,1],X,d)$ the space of 
$d$-Lipschitz curves $c:[0,1] \rightarrow X$. Let also $\displaystyle 
l_{d}$ denote the length functional associated to the distance $d$.

\subsection{Gamma-convergence of length functionals}

\begin{definition}
For any $\varepsilon \in (0,1)$ we define the {\bf length functional }
$$l_{\varepsilon}: \mathcal{L}_{\varepsilon}(X,d,  \delta) \rightarrow
[0,+\infty] \quad , \quad l_{\varepsilon}(x,c) \ = \ l^{x}_{\varepsilon}(c) \ = \ \frac{1}{\varepsilon}
\, 
l_{d}(\delta^{x}_{\varepsilon} c)  $$
The domain of definition of the functional 
$\displaystyle l_{\varepsilon}$ is the space: 
 $$\mathcal{L}_{\varepsilon}(X,d,  \delta) \ = \ \left\{(x ,c) \in X  
 \times \mathcal{C}([0,1],X) 
\mbox{ : } c: [0,1] \in U(x) \, \, , \,  \right.$$ 
$$\left. \delta^{x}_{\varepsilon}c \mbox{ is } 
d-Lip \mbox{ and } Lip(\delta^{x}_{\varepsilon}c)\,  \leq \, 2  \, 
l_{d}(\delta^{x}_{\varepsilon}c) 
\right\} $$
\label{thespaceleps}
\end{definition}

The last condition from the definition of $\displaystyle 
\mathcal{L}_{\varepsilon}(X,d,  \delta)$ is a selection of parameterization 
of the path $c([0,1])$. Indeed, by the reparameterization theorem, if 
$\displaystyle \delta^{x}_{\varepsilon} c :[0,1] \rightarrow (X,d)$ is a 
$d$-Lipschitz curve of length $\displaystyle L = l_{d}(\delta^{x}_{\varepsilon}c)$ 
then $\displaystyle \delta^{x}_{\varepsilon}c([0,1])$ can be reparameterized by length, that is there exists a 
increasing  function $\phi:[0,L] \rightarrow [0,1]$ such that $\displaystyle 
c'= \delta^{x}_{\varepsilon} c\circ \phi$ is 
a $d$-Lipschitz curve with $Lip(c') \leq 1$. But we can use a second affine
reparameterization which sends $[0,L]$ back to $[0,1]$ and we get a Lipschitz
curve $c"$ with $c"([0,1]) = c'([0,1])$ and $\displaystyle Lip(c") \leq 2 
l_{d}(c)$.

We shall use the following definition of Gamma-convergence (see the book
 \cite{dalmaso} for the notion of Gamma-convergence). Notice the use of
convergence of sequences only in the second part of the definition.

\begin{definition}
Let $Z$ be a metric space with distance function $D$ and $\displaystyle \left(
l_{\varepsilon}\right)_{\varepsilon > 0}$ be a family of functionals $\displaystyle 
l_{\varepsilon}: Z_{\varepsilon} \subset Z \rightarrow [0,+\infty]$. Then 
$\displaystyle l_{\varepsilon}$ {\bf Gamma-converges} to the functional 
$\displaystyle l: Z_{0} \subset Z \rightarrow [0,+\infty]$ if: 
\begin{enumerate}
\item[(a)] ({\bf liminf inequality}) for any function $\displaystyle \varepsilon \in
(0,\infty)  \mapsto 
x_{\varepsilon} \in Z_{\varepsilon}$ such that $\displaystyle \lim_{\varepsilon
\rightarrow 0} x_{\varepsilon} \, = \, x_{0} \in Z_{0}$ we have 
$$l(x_{0}) \, \leq \, \liminf_{\varepsilon \rightarrow 0}
l_{\varepsilon}(x_{\varepsilon})$$
\item[(b)] ({\bf existence of a recovery sequence}) For any $\displaystyle x_{0} \in Z_{0}$ 
and for any sequence $\displaystyle \left( \varepsilon_{n} \right)_{n \in \mathbb{N}}$
such that $\displaystyle \lim_{n \rightarrow \infty} \varepsilon_{n} \, = \, 0$ there
is a sequence $\displaystyle \left( x_{n} \right)_{n \in \mathbb{N}}$ with
$\displaystyle x_{n} \in Z_{\varepsilon_{n}}$ for any $n \in \mathbb{N}$, such that 
$$l(x_{0}) \, = \, \lim_{n \rightarrow \infty}
l_{\varepsilon_{n}}(x_{n})$$
\end{enumerate}
\end{definition}

We shall take as the metric space $Z$ the space
 $X \times \mathcal{C}([0,1],X)$ with the distance 
 $$ D((x,c), (x', c')) \ = \ \max\left\{ d(x, x') \, , \, \sup \left\{ d(c(t),
 c'(t)) \mbox{ : } t \in [0,1] \right\} \right\} $$

Let  $\displaystyle \mathcal{L}(X,d,\delta)$be the class of all 
$\displaystyle (x ,c) \in X  \times \mathcal{C}([0,1],X)$ which appear as  limits  
$\displaystyle (x_{n}, c_{n}) \rightarrow (x,c)$, with $\displaystyle 
(x_{n}, c_{n}) \in \mathcal{L}_{\varepsilon_{n}}(X,d,  \delta)$, 
the family $\displaystyle (c_{n})_{n}$ is $d$-equicontinuous and $\displaystyle 
\varepsilon_{n} \rightarrow 0$ as $n \rightarrow \infty$.

\begin{definition}
A triple $(X,d,\delta)$ is a {\bf length dilatation structure} if $(X,d)$ is a
complete, locally compact metric space such that A0, A1, 
A2,  are satisfied, together with  the following axioms: 
\begin{enumerate}
\item[\bf{A3L.}] there is a functional $\displaystyle l : \mathcal{L}(X,d,  \delta) \rightarrow
[0,+\infty]$ such that for any $\varepsilon_{n} \rightarrow 0$ as $n \rightarrow
\infty$ the sequence of functionals 
$\displaystyle l_{\varepsilon_{n}}$ Gamma-converges to the functional $l$. 
\item[\bf{A4+}] Let us define 
$\displaystyle \Delta^{x}_{\varepsilon}(u,v) =
\delta_{\varepsilon^{-1}}^{\delta^{x}_{\varepsilon} u} \delta^{x}_{\varepsilon}
v$ and $\displaystyle \Sigma^{x}_{\varepsilon}(u,v) = 
\delta_{\varepsilon^{-1}}^{x} \delta_{\varepsilon}^{\delta^{x}_{\varepsilon} u}
v$. 
Then we have the limits 
$$\lim_{\varepsilon \rightarrow 0}  \Delta^{x}_{\varepsilon}(u,v) =  \Delta^{x}(u, v)  $$
$$\lim_{\varepsilon \rightarrow 0}  \Sigma^{x}_{\varepsilon}(u,v) =  \Sigma^{x}(u, v)  $$
uniformly with respect to $x, u, v$ in compact set. 
\end{enumerate} 
\label{deflds}
\end{definition}

\begin{remark}
For strong dilatation structures the axioms A0 - A4 imply A4+. The
transformations $\displaystyle \Sigma^{x}_{\varepsilon}(u, \cdot)$ have the
interpretation of approximate left translations in the tangent space of $(X,d)$
at $x$.  
\end{remark}

For any $\varepsilon \in (0,1)$ and any $x \in X$ the length functional 
$\displaystyle l^{x}_{\varepsilon}$ induces a distance on $U(x)$: 
$$ \mathring{d}^{x}_{\varepsilon}(u,v) \ = \ \inf\left\{ l^{x}_{\varepsilon}(c)
\mbox{ : } (x,c) \in \mathcal{L}_{\varepsilon}(X,d,  \delta) \, , \, c(0) = u \,
, \, c(1) = v \right\} $$
In the same way the length functional $l$ from A3L induces a distance
$\displaystyle \mathring{d}^{x}$ on $U(x)$.

Gamma-convergence implies that 
\begin{equation}
\mathring{d}^{x}(u,v) \, \geq \, \limsup_{\varepsilon \rightarrow 0} \mathring{d}^{x}_{\varepsilon}(u,v)
\label{dsup}
\end{equation}

\begin{remark}
Without supplementary hypotheses we cannot prove A3 from A3L, that is 
in principle length dilatation structures are not strong dilatation structures. 
\label{rkused}
\end{remark}

\section{The Radon-Nikodym property}
\label{radon}
 
\subsection{Differentiability with respect to dilatation structures}

For any strong dilatation structure or length dilatation structure there is an associated  notion  of differentiability (section 7.2 \cite{buligadil1}). 
First we need the definition of a morphism of conical groups. 

\begin{definition}
 Let $(N,\delta)$ and $(M,\bar{\delta})$ be two  conical groups. A function $f:N\rightarrow M$ is a conical group morphism if $f$ is a group morphism and for any $\varepsilon>0$ and $u\in N$ we have 
 $\displaystyle f(\delta_{\varepsilon} u) = \bar{\delta}_{\varepsilon} f(u)$. 
\label{defmorph}
\end{definition}

The definition of the derivative, or differential,  with respect to dilatations structures follows. In the case of a pair of Carnot groups this is just the definition of the Pansu derivative 
introduced in \cite{pansu}.

 \begin{definition}
 Let $(X, d, \delta)$ and $(Y, \overline{d}, \overline{\delta})$ be two 
 strong dilatation structures  or length and $f:X \rightarrow Y$ be a continuous function. The function $f$ is differentiable in $x$ if there exists a 
 conical group morphism  $\displaystyle D \, f(x):T_{x}X\rightarrow T_{f(x)}Y$, defined on a neighbourhood of $x$ with values in  a neighbourhood  of $f(x)$ such that 
\begin{equation}
\lim_{\varepsilon \rightarrow 0} \sup \left\{  \frac{1}{\varepsilon} \overline{d} \left( f\left( \delta^{x}_{\varepsilon} u\right) ,  \overline{\delta}^{f(x)}_{\varepsilon} D \, f(x)  (u) \right) \mbox{ : } d(x,u) \leq \varepsilon \right\}Ê  = 0 , 
\label{edefdif}
\end{equation}
The morphism $\displaystyle D \, f(x) $ is called the derivative, or differential,  of $f$ at $x$.

\label{defdiffer}
\end{definition}

The definition also makes sense if the function $f$ is defined on a open subset of $(X,d)$.

\subsection{The Radon-Nikodym property}
 
 \begin{definition} 
A strong dilatation structure or a length dilatation structure
  has the {\bf Radon-Nikodym property (or rectifiability
 property, or RNP)} if  any  
Lipschitz curve $\displaystyle c : [a,b] \rightarrow (X,d)$ is derivable 
almost everywhere. 
\label{defrn}
 \end{definition}

\subsection{Two examples}
\label{2ex}

The following two easy examples will 
show that not any strong dilatation structure has the Radon-Nikodym property.

For  $\displaystyle (X,d)  =  ( \mathbb{V}, d)$, a real, finite dimensional,
normed vector space, with distance $d$ induced by the norm, the (usual) 
 dilatations $\displaystyle \delta^{x}_{\varepsilon}$ are given by:  
$$ \delta_{\varepsilon}^{x} y \ = \ x + \varepsilon (y-x) $$
Dilatations are defined everywhere.

There are few things to check:  axioms 0,1,2 are obviously 
true. For axiom A3, remark that for any $\varepsilon > 0$, $x,u,v \in X$ we 
have: 
$$\frac{1}{\varepsilon} d(\delta^{x}_{\varepsilon} u , 
\delta^{x}_{\varepsilon} v ) \ = \ d(u,v) \ , $$
therefore for any $x \in X$ we have $\displaystyle d^{x} = d$. 

Finally, let us check the axiom A4. For any $\varepsilon > 0$ and $x,u,v \in X$ we have
$$\delta_{\varepsilon^{-1}}^{\delta_{\varepsilon}^{x} u} \delta_{\varepsilon}^{x} v \ = \ 
x + \varepsilon  (u-x) + \frac{1}{\varepsilon} \left( x+ \varepsilon(v-x) - x - \varepsilon(u-x) \right) \ = \ $$
$$ = \ x + \varepsilon  (u-x) + v - u$$ 
therefore this quantity converges to 
$$x + v - u \ = \ x + (v - x) - (u - x)$$
as $\varepsilon \rightarrow 0$. The axiom A4 is verified. 

This dilatation structure has the Radon-Nikodym property.

Further is an example of a dilatation structure which does not have the
Radon-Nikodym property.  Take $\displaystyle X = \mathbb{R}^{2}$ with the euclidean 
distance $\displaystyle d$. For any $z \in \mathbb{C}$ of the 
form $z= 1+ i \theta$ we define dilatations 
$$\delta_{\varepsilon} x = \varepsilon^{z} x  \ .$$
It is easy to check that $\displaystyle (\mathbb{R}^{2},d, \delta)$ 
is a dilatation structure, with  dilatations 
$$\delta^{x}_{\varepsilon} y = x + \delta_{\varepsilon} (y-x)  $$

Two such dilatation structures (constructed with the help of complex numbers 
$1+ i \theta$ and $1+ i \theta'$) are equivalent if and only if $\theta = \theta'$.  

There are two other interesting  properties of these dilatation structures. 
The first is that if $\theta \not = 0$ then there are no non trivial 
Lipschitz curves in $X$ which are differentiable almost everywhere. It means
that such dilatation structure does not have the Radon-Nikodym property. 

The second property is that any holomorphic and Lipschitz function from $X$ to $X$ (holomorphic in the 
usual sense on $X = \mathbb{R}^{2} = \mathbb{C}$) is differentiable almost everywhere, but there are 
Lipschitz functions from $X$ to $X$ which are not differentiable almost everywhere (suffices to take a 
$\displaystyle \mathcal{C}^{\infty}$ function from  $\displaystyle \mathbb{R}^{2}$ to $\displaystyle \mathbb{R}^{2}$ which is not holomorphic).

\subsection{Length formula from Radon-Nikodym property}

\begin{definition}
In a normed conical group $N$   we shall denote by $D(N)$ the set of all 
$u\in N$ with the property that $\displaystyle 
\varepsilon \in ((0,\infty),+) \mapsto \delta_{\varepsilon} u \in N$ is a morphism of
groups.
\label{defdisn}
\end{definition}
$D(N)$ is always non empty, because it contains the neutral element of $N$. 
$D(N)$ is also a cone, with dilatations $\displaystyle   \delta_{\varepsilon}$, 
and a closed set.

 \begin{proposition}
 Let $(X,d,\delta)$ be a strong dilatation structure. Then the following are
 equivalent: 
 \begin{enumerate}
 \item[(a)] $(X,d,\delta)$ has the Radon-Nikodym property; 
 \item[(b)] for any  Lipschitz curve 
  $\displaystyle c : [a, b] \rightarrow (X,d)$, 
  for almost every $t \in [a,b]$ there is 
  $\displaystyle \dot{c}(t) \in D(T_{c(t}(X,d,\delta))$ such that 
$$\frac{1}{\varepsilon} d(c(t+\varepsilon) , \delta_{\varepsilon}^{c(t)} \dot{c}(t)) \rightarrow 0 $$
$$\frac{1}{\varepsilon} d(c(t-\varepsilon) , \delta_{\varepsilon}^{c(t)} 
inv^{c(t)}( \dot{c}(t))) \rightarrow 0  $$
\end{enumerate}
\end{proposition}

\begin{proof}  
 It is  straightforward that a conical group morphism $f: \mathbb{R} \rightarrow N$ is defined by its value $f(1)\in N$. Indeed, for any $a>0$ we have $\displaystyle f(a) = \delta_{a} f(1)$ and for any 
 $a<0$ we have $\displaystyle f(a) = \delta_{a} f(1)^{-1}$. From the morphism property we also 
 deduce that  
 $$\delta v = \left\{ \delta_{a} v \mbox{ : } a>0 , v=f(1) \mbox{ or } v=f(1)^{-1} \right\}$$
 is a one parameter group and that for all $\alpha, \beta >0$ we have 
 $\displaystyle \delta_{\alpha+\beta} u = \delta_{\alpha}u \,  \delta_{\beta}u$. We have
 therefore a bijection between conical group morphisms $f: \mathbb{R} \rightarrow
 (N,\delta)$  and elements of $D(N)$. 
 
 A Lipschitz curve 
 $\displaystyle c : [a,b] \rightarrow (X,d)$ is derivable in $t \in (a,b)$ if 
 and only if there is a morphism of normed conical groups 
 $\displaystyle f: \mathbb{R} \rightarrow T_{c(t}(X,d,\delta)$ such that 
 for any $a \in \mathbb{R}$ we have 
 $$\lim_{\varepsilon \rightarrow 0} \frac{1}{\varepsilon} \, d(c(t+\varepsilon a), 
 \delta^{c(t)}_{\varepsilon} f(a)) \ = \ 0$$
 Take  $\displaystyle \dot{c}(t) = f(1)$. Then $\displaystyle \dot{c}(t) \in 
 D(T_{c(t}(X,d,\delta))$. For any $a > 0$ we have $\displaystyle f(a) =
 \delta^{c(t)}_{a} \dot{c}(t)$; otherwise if $a < 0$ we have 
 $\displaystyle f(a) =
 \delta^{c(t)}_{a} \, inv^{c(t)} \,\dot{c}(t)$. This implies the equivalence 
 stated on the proposition.
 \end{proof}

 \begin{theorem}
Let $(X,d,\delta)$ be a strong dilatation structure with the Radon-Nikodym property, 
over a complete length metric space $(X,d)$. Then   for any $x, y \in X$ we have 
$$d(x,y) \ = \ \inf \left\{ \int_{a}^{b} d^{c(t)}(c(t),\dot{c}(t)) \mbox{ d}t  \mbox{ :
} c:[a,b]\rightarrow X \mbox{ Lipschitz }, \right. $$
$$\left.  c(a) = x , c(b) = y \right\}  $$
\label{fleng}
\end{theorem}

\begin{proof}
From theorem \ref{tupper} we deduce that for almost every $t\in(a,b)$ 
the upper dilatation of $c$ in $t$ can be expressed as:
$$Lip(c)(t) = \lim_{s\rightarrow t} \frac{d(c(s),c(t))}{\mid s-t \mid} $$
 
If the dilatation structure has the Radon-Nikodym property then for almost every $t \in [a,b]$ there is 
$\displaystyle \dot{c}(t) \in D(T_{c(t)} X)$ such that 
$$\displaystyle \frac{1}{\varepsilon} d(c(t+\varepsilon) , \delta_{\varepsilon}^{c(t)} \dot{c}(t)) 
\rightarrow 0$$ 
Therefore for almost every $t \in [a,b]$ we have
$$\displaystyle Lip(c)(t) = \lim_{\varepsilon\rightarrow 0} 
\frac{1}{\varepsilon} d(c(t+\varepsilon),c(t)) = d^{c(t)}(c(t),\dot{c}(t))$$ 
The formula for length follows from here. 
\end{proof}

A straightforward consequence is that the distance $d$ is uniquely determined by the 
"distribution" $\displaystyle x \in X \mapsto D(T_{x}(X,d,\delta))$ and the 
function which associates to any $x \in X$ the "norm" $\displaystyle 
\| \cdot \|_{x} : D(T_{x}(X,d,\delta)) \rightarrow [0,+ \infty)$. 
  
 \begin{corollary}
Let $(X,d,\delta)$ and $(X,\bar{d},\bar{\delta})$ be two strong dilatation structures 
with the Radon-Nikodym property , which are also complete length metric spaces, 
such that for any $x \in X$ we have  $\displaystyle D(T_{x}(X,d,\delta)) =
D(T_{x}(X,d',\delta'))$ and $\displaystyle d^{x}(x,u ) = \bar{d}^{x}(x, u)$ for any 
$\displaystyle u \in  D(T_{x}(X,d,\delta))$. Then $d = d'$.
 \end{corollary}

\subsection{Equivalent dilatation structures and their distributions}

\begin{definition}
{\bf Two strong dilatation structures} $(X, \delta , d)$ and $(X,
\overline{\delta} , \overline{d})$  {\bf are equivalent}  if 
\begin{enumerate}
\item[(a)] the identity  map $\displaystyle id: (X, d) \rightarrow (X, \overline{d})$ is bilipschitz and 
\item[(b)]  for any $x \in X$ there are functions $\displaystyle P^{x}, Q^{x}$ (defined for $u \in X$ sufficiently close to $x$) such that  
\begin{equation}
\lim_{\varepsilon \rightarrow 0} \frac{1}{\varepsilon} \overline{d} \left( \delta^{x}_{\varepsilon} u ,  \overline{\delta}^{x}_{\varepsilon} Q^{x} (u) \right)  = 0 , 
\label{dequiva}
\end{equation}
\begin{equation}
 \lim_{\varepsilon \rightarrow 0} \frac{1}{\varepsilon} d \left( \overline{\delta}^{x}_{\varepsilon} u ,  
 \delta^{x}_{\varepsilon} P^{x} (u) \right)  = 0 , 
\label{dequivb}
\end{equation}
uniformly with respect to $x$, $u$ in compact sets. 
\end{enumerate}
\label{dilequi}
\end{definition}

\begin{proposition}
 $(X, \delta , d)$ and $(X, \overline{\delta} , \overline{d})$  are equivalent  if and 
only if 
\begin{enumerate}
\item[(a)] the identity  map $\displaystyle id: (X, d) \rightarrow (X,
\overline{d})$ is bilipschitz, 
\item[(b)]  for any $x \in X$ there are conical group morphisms: 
 $$\displaystyle P^{x}: T_{x}(X, \overline{\delta} , \overline{d}) 
 \rightarrow T_{x} (X, \delta , d) \mbox{ and } \displaystyle  Q^{x}: T_{x} (X, \delta , d) \rightarrow 
 T_{x}(X, \overline{\delta} , \overline{d})$$
  such that the following limits exist  
\begin{equation}
\lim_{\varepsilon \rightarrow 0}  \left(\overline{\delta}^{x}_{\varepsilon}\right)^{-1}  \delta^{x}_{\varepsilon} (u) = Q^{x}(u) , 
\label{dequivap}
\end{equation}
\begin{equation}
 \lim_{\varepsilon \rightarrow 0}  \left(\delta^{x}_{\varepsilon}\right)^{-1}  \overline{\delta}^{x}_{\varepsilon} (u) = P^{x}(u) , 
\label{dequivbp}
\end{equation}
and are uniform with respect to $x$, $u$ in compact sets. 
\end{enumerate}
\label{pdilequi}
\end{proposition}

 The next theorem shows a link between the tangent bundles of equivalent dilatation structures. 
 
 \begin{theorem} 
 Let $(X, d, \delta)$ and $(X, \overline{d}, \overline{\delta})$  be  equivalent
 strong  dilatation structures.  Then for any $x \in X$ and 
 any $u,v \in X$ sufficiently close to $x$ we have:
 \begin{equation}
 \overline{\Sigma}^{x}(u,v) = Q^{x} \left( \Sigma^{x} \left( P^{x}(u) , P^{x}(v) \right)\right) . 
 \label{isoequiv}
 \end{equation}
 The two tangent bundles  are therefore isomorphic in a natural sense. 
 \label{tisoequiv}
 \end{theorem}
 As a consequence, the following corollary is straightforward. 
 
\begin{corollary}
Let  $(X, d, \delta)$ and $(X, \overline{d}, \overline{\delta})$  be  
equivalent strong dilatation structures. Then for any $x \in X$ we have 
$$\displaystyle Q^{x} (D(T_{x}(X, \delta , d))) \ = \ D(T_{x}(X, \overline{\delta} ,
\overline{d}))  $$

If  $(X, d, \delta)$ has the Radon-Nikodym property , then 
$(X, \overline{d}, \overline{\delta})$ has the same property. 

Suppose that $(X, d, \delta)$ and $(X, \overline{d}, \overline{\delta})$  are 
complete length spaces with the Radon-Nikodym property . If the functions 
$\displaystyle P^{x}, Q^{x}$ from definition \ref{dilequi} (b) are isometries, 
then $\displaystyle d = \overline{d}$. 
\end{corollary}

\section{Tempered dilatation structures}

The notion of a tempered dilatation structure is inspired by 
the results from Venturini \cite{venturini} and Buttazzo, De Pascale and 
Fragal\`a \cite{buttazzo1}. 

The examples of length dilatation structures from this section are provided 
by the extension of some results from \cite{buttazzo1} (propositions 2.3, 2.6 
and a part of theorem 3.1) to dilatation 
structures.

We recall some definition2.1 from \cite{buttazzo1} section 2. Let $\Omega$ be a 
given connected open subset of $\displaystyle \mathbb{R}^{N}$ endowed with 
 the  distance induced by the euclidean norm. Given two positive constants
  $c < C$, let $ D(\Omega)$ be the class of all length distances on $\Omega$ 
  such that 
 \begin{equation}
c \|u - v\| \ \leq  \, d(u, v) \, \leq \,  C \|u - v\| 
\label{but1}
\end{equation}
for all $u, v \in \Omega$. We suppose that $D(\Omega$ is not empty. $D(\Omega)$
is endowed with the topology of uniform convergence on compact subsets of
$\Omega \times \Omega$. 

To any $d \in D(\Omega)$ is associated the function 
$$\phi_{d}(x, u) \, = \, \limsup_{\varepsilon \rightarrow 0}
\frac{1}{\varepsilon} \, d(x, x + \varepsilon u) $$
This function is measurable in $x$ and convex positively one-homogeneous in $z$.
From (\ref{but1}) we see that $\displaystyle \phi_{d}$ has the property: 
\begin{equation}
c \|z\| \ \leq  \, \phi_{d}(x,z) \, \leq \,  C \|z\| 
\label{but2}
\end{equation}
By proposition 2.4 \cite{buttazzo1} the function $\displaystyle \phi_{d}$ allows
to write in an integral form he length functional $\displaystyle l_{d}$
associated to $d$: for any lipschitz curve $c: [0,1] \rightarrow \Omega$ we have
$$l_{d}(c) \, = \, \int_{0}^{1} \phi_{d}(c(t), \dot{c}(t)) \mbox{ d}t$$
With the first example of a dilatation structure with the Radon-Nikodym in mind 
(see subsection \ref{2ex}), we can easily rewrite this in terms of dilatation
structures. Indeed, it suffices to replace $\mathbb{R}^{N}$ with the euclidean
distance by a metric space $(X,\bar{d})$ endowed with a dilatation structure. 
Then we may choose to see the $\displaystyle \|z\|$ from relation 
(\ref{but2}) as the distance $\displaystyle \bar{d}^{x}(u,z)$. Finally, instead
of (\ref{but1}), (\ref{but2}), we may write $d \in mathcal{D}(\Omega,
\bar{d}, \bar{\delta})$ if 
$$c \, \bar{d}^{x}(u,v) \, \leq \, \frac{1}{\varepsilon} \,
d(\bar{\delta}^{x}_{\varepsilon} u , \bar{\delta}^{x}_{\varepsilon} v ) \, \leq 
\, C \, \bar{d}^{x}(u,v) $$
and $\displaystyle \phi_{d}$ could also be rewritten as: 
$$\phi_{d}(x, u) \, = \, \limsup_{\varepsilon \rightarrow 0}
\frac{1}{\varepsilon} \, d(x, \delta^{x}_{\varepsilon} u ) $$
The euclidean distance, or the distance $\displaystyle \bar{d}$ is 
here fixed, and the class $D(\Omega)$ is defined relatively to $\displaystyle 
\bar{d}$. Remark that for $\displaystyle \bar{d}$ being the euclidean distance 
in $\displaystyle X = \mathbb{R}^{N}$, it is true that $\displaystyle 
\bar{d} \in \mathcal{D}(\Omega, \bar{d}, \bar{\delta})$. This inspired us to call
such dilatation structures "tempered".

The construction is presented further in detail.  
The following definition gives a class of distances $\mathcal{D}(\Omega,
\bar{d}, \bar{\delta})$, associated to a strong dilatation structure 
$(\Omega, \bar{d}, \bar{\delta})$, which in some sense generalizes the class of distances 
$\mathcal{D}(\Omega)$ from \cite{buttazzo1}, definition 2.1. 

\begin{definition}
For any strong dilatation structure $(\Omega, \bar{d}, \bar{\delta})$   and
constants $0 < c < C$ we define the class $\mathcal{D}(\Omega,
\bar{d}, \bar{\delta})$ of all distance functions $d$ on $\Omega$ such that 
\begin{enumerate}
\item[(a)] $d$ is a length distance, 
\item[(b)] for any $\varepsilon > 0$ and any $x, u, v$ sufficiently 
closed we have: 
\begin{equation}
c \, \bar{d}^{x}(u,v) \, \leq \, \frac{1}{\varepsilon} \,
d(\bar{\delta}^{x}_{\varepsilon} u , \bar{\delta}^{x}_{\varepsilon} v ) \, \leq 
\, C \, \bar{d}^{x}(u,v) 
\label{new2.3}
\end{equation}
\end{enumerate}
The dilatation structure $(\Omega, \bar{d}, \bar{\delta})$ is {\bf tempered} if 
there are constants $c, C$ such that $\bar{d} \in \mathcal{D}(\Omega,
\bar{d}, \bar{\delta})$. 

On $\mathcal{D}(\Omega, \bar{d}, \bar{\delta})$ we put the topology of uniform
convergence (induced by distance $\bar{d}$) on compact subsets of $\Omega \times
\Omega$. 
\label{dtempered}
\end{definition}

To any distance $d \in \mathcal{D}(\Omega, \bar{d}, \bar{\delta})$ we associate
the function: 
$$\phi_{d}(x, u) \, = \, \limsup_{\varepsilon \rightarrow 0}
\frac{1}{\varepsilon} \, d(x, \delta^{x}_{\varepsilon} u ) $$
defined for any $x, u \in \Omega$ sufficiently close. We have therefore 
\begin{equation}
c \, \bar{d}^{x}(x, u) \, \leq \, \phi_{d}(x,u) \, \leq \, C \, \bar{d}^{x}(x,u)
\label{new2.6}
\end{equation}

Notice that if $d \in \mathcal{D}(\Omega, \bar{d}, \bar{\delta})$ then for any 
$x, u, v$ sufficiently close we have 
$$- \bar{d}(x,u) \, O(\bar{d}(x,u)) \, + \, c \,  \bar{d}^{x}(u,v) \, \leq $$
$$ \leq 
\, d(u,v) \, \leq \,   \, C \, \bar{d}^{x}(u,v) \, + \, 
 \bar{d}(x,u) \, O(\bar{d}(x,u))$$

If $c: [0,1] \rightarrow \Omega$ is a $d$-Lipschitz curve and 
$d \in \mathcal{D}(\Omega, \bar{d}, \bar{\delta})$ then we may decompose it 
in a finite family of curves $\displaystyle c_{1}, ... , c_{n}$ (with $n$ depending on $c$) 
such that there are $\displaystyle x_{1}, ... , x_{n} \in \Omega$ with 
$\displaystyle c_{k} $ is $\displaystyle \bar{d}^{x_{k}}$-Lipschitz. Indeed, the image of the 
curve $c([0,1])$ is compact, therefore we may cover it with a finite number of
balls $\displaystyle B(c(t_{k}), \rho_{k}, \bar{d}^{c(t_{k})})$ and apply 
(\ref{new2.3}). If moreover $(\Omega, \bar{d}, \bar{\delta})$ is tempered then 
it follows that $c: [0,1] \rightarrow \Omega$  $d$-Lipschitz curve is equivalent
with $c$ $\bar{d}$-Lipschitz curve. 

By using the same arguments as in the proof of theorem \ref{fleng}, we get the
following extension of proposition 2.4 \cite{buttazzo1}. 

\begin{proposition}
If $(\Omega, \bar{d}, \bar{\delta})$ is tempered, with the Radon-Nikodym 
property, and $d \in \mathcal{D}(\Omega, \bar{d}, \bar{\delta})$ then 
$$d(x,y) \ = \ \inf \left\{ \int_{a}^{b} \phi_{d}(c(t),\dot{c}(t)) \mbox{ d}t  \mbox{ :
} c:[a,b]\rightarrow X \mbox{ $\bar{d}$-Lipschitz }, \right. $$
$$\left.  c(a) = x , c(b) = y \right\}  $$
\label{new2.4}
\end{proposition}

The next theorem is a generalization of the implication (i) $\Rightarrow$ (iii), theorem 
 3.1 \cite{buttazzo1}.

\begin{theorem}
Let $(\Omega, \bar{d}, \bar{\delta})$ be a strong dilatation structure which is 
tempered, with the Radon-Nikodym 
property, and $\displaystyle d_{n} \in \mathcal{D}(\Omega, \bar{d}, \bar{\delta})$ 
 a sequence  of distances converging to $d \in \mathcal{D}(\Omega, \bar{d},
 \bar{\delta})$. Denote by $\displaystyle L_{n}, L$ the length functional induced
 by the distance $\displaystyle d_{n}$, respectively by $d$. 
 Then $\displaystyle L_{n}$ $\Gamma$-converges to $L$. 
 \label{new3.1}
 \end{theorem}
 
 \begin{proof}
 The proof (\cite{buttazzo1} p. 252-253) is almost identical, we only need to 
replace everywhere expressions like  $\mid x - y\mid$ by $\bar{d}(x,y)$ and 
use proposition \ref{new2.4}, relations (\ref{new2.6}) and (\ref{new2.3})
instead of respectively proposition 2.4 and relations (2.6) and (2.3)
\cite{buttazzo1}. 
 \end{proof}

Using this result we obtain a large class of examples of length dilatation
structures.

\begin{corollary}
If  $(\Omega, \bar{d}, \bar{\delta})$ is a strong dilatation structure which is 
tempered and it has  the Radon-Nikodym property then it is a length dilatation structure. 
\label{cortemp}
\end{corollary}

\begin{proof}
Indeed, from the hypothesis we deduce that $\displaystyle
\bar{\delta}^{x}_{\varepsilon} \bar{d} \, \in \, 
 \mathcal{D}(\Omega, \bar{d}, \bar{\delta})$. For any sequence 
 $\displaystyle \varepsilon_{n} \rightarrow 0$ we thus obtain a sequence of 
 distances $\displaystyle d_{n} \, = \, \bar{\delta}^{x}_{\varepsilon_{n}} 
 \bar{d}$ converging to $\bar{d}^{x}$. We apply now theorem \ref{new3.1} 
 and we get the result. 
\end{proof}

\end{document}